\documentclass{amsart}
	\pdfoutput=1
	\usepackage{aliascnt}
	\usepackage[colorlinks=true, linkcolor=blue, citecolor=blue]{hyperref}
	\usepackage{amssymb}	
	\usepackage{amsmath}

	\newaliascnt{lemma}{thm}
	  
	\aliascntresetthe{lemma}

	\newaliascnt{prop}{thm}
	\newtheorem{prop}[prop]{Proposition} 
	\aliascntresetthe{prop}

	\newaliascnt{defn}{thm}
	
	\aliascntresetthe{defn}

	\newaliascnt{cor}{thm}
	
	\aliascntresetthe{cor}

	\newaliascnt{rem}{thm}
	\newtheorem{rem}[rem]{Remark}
	\aliascntresetthe{rem}

	\newaliascnt{exm}{thm}
	\newtheorem{exm}[exm]{Example}
	\aliascntresetthe{exm}

	\usepackage{layout}
	\newcommand{\K}{\mathbb{K}}
	\newcommand{\Ps}{\mathbb{P}}
	\def\cocoa{\href{http://cocoa.dima.unige.it}{\hbox{\rm C\kern-.13em o\kern-.07em C\kern-.13em o\kern-.15em A}}}

\begin{document}
\title[A remark on Waring decompositions of some special plane quartics]{A remark on Waring decompositions\\of some special plane quartics}
\author{Alessandro De Paris}
\date{}
\begin{abstract}
This work concerns Waring decompositions of a certain kind of plane quartics of high rank. The main result is the following. Let $x,l_1,\ldots,l_7$ be linear forms and $q$ a quadratic form on a vector space of dimension $3$. If $x^2q=l_1^4+\cdots +l_7^4$ and the lines $l_1=0$, $\ldots$, $l_7=0$ in $\Ps^2$ intersect $x=0$ in seven distinct points, then the line $x=0$ is (possibly improperly) tangent to the conic $q=0$.
\end{abstract}
\maketitle

\section{Introduction}
A \emph{Waring decomposition} of a homogeneous polynomial $f\in\K[x_0,\ldots ,x_n]$ of degree $d$ is given by a sum of $d$-th powers of linear forms:
\[
f=l_1^d+\cdots l_r^d\;.
\]
The minimum number $r$ of terms for such a decomposition is called the \emph{Waring rank}, or \emph{symmetric tensor rank}, of $f$. In this paper we shall refer to it simply as \emph{the rank of $f$}. The so-called \emph{little Waring problem} for polynomials asks for the maximum possible rank, given $d$ and $n$. We refer the reader to \cite{G} for a friendly exposition about this problem, and to \cite{L} for an extensive and up-to-date survey on the role of tensor rank theory in a broad range of applications. The latter will also be our default reference for basic terminology.

When $n=1$ the answer is known: a detailed description is given by the Comas-Seiguer theorem (see \cite[Theorem~9.2.2.1]{L}). In the case $(d,n)=(3,2)$ the maximal rank is five (see \cite[Section~8]{LT}). A careful account on rank stratification in the case $(d,n)=(4,2)$, is given by \cite[Theorem~44]{BGI} (cf.~also \cite[Theorem~10.9.3.2]{L}), among a considerable amount of other interesting results. However, in what concerns maximal rank there is lack of completeness, and, to our knowledge, the answer for $(d,n)=(4,2)$ remains unknown.

According to \cite[0.2, p.~xv]{L}, `Results [$\ldots$] indicate there is beautiful geometry associated to rank that is only beginning to be discovered'. In such situations, we believe that even to point out some simple phenomena may contribute to the discovery. In this paper we turn our attention to quartics that, in geometric terms, are decomposed in a double line and a conic. They seem to us a good test, because for plane cubics the maximal rank is achieved by a conic together with a tangent line to it. We collect some remarks about the space of all Waring decompositions of such quartics. In view of the above said, we hope that the latter of them (\autoref{Mr}) is sufficiently interesting to be brought to the attention of researchers in the field.

\section{Preliminary remarks}

First of all,
\begin{center}
we fix a field $\K$ of zero characteristic
\end{center}
(for applicative purposes, generally one may assume $\K=\mathbb{R}$ or $\mathbb{C}$). That is a common assumption when one deals with rank of polynomials, maybe because positive characteristics often lead to cumbersome subtleties (for instance, note that if the characteristic is $2$, then $xy$ can not be expressed as a sum of squares of linear forms in $x,y$). However, we shall outline how things go in positive characteristic by making a quick separate remark.

We start with some very elementary considerations, in the usual coordinate settings. We do not assume $\K$ to be algebraically closed: that hypothesis will be needed only if one wants to directly connect the following starting remarks with the little Waring problem, and to provide some general qualitative considerations we are going to make with a more precise meaning.

Let us consider six linear forms in a polynomial ring $\K\left[x_0,x_1,x_2\right]$:
\begin{equation}
\label{l1-6}
l_1:=x_0+h_1x_1+k_1x_2\;,\ldots,\; l_6:=x_0+h_6x_1+k_6x_2\;.
\end{equation}
We seek for a linear combination of $l_1^4,\ldots ,l_6^4$ such that the quartic curve it represents doubly contains the line $x_2=0$:
\begin{equation}\label{G}
\alpha_1l_1^4+\ldots +\alpha_6l_6^4=x_2^2q\left(x_0,x_1,x_2\right)\;,
\end{equation}
with $q$ a quadratic form. A simple count of dimensions indicates that for a generic choice of $l_1,\ldots ,l_6$ one does not expect such a combination to exist: in geometric terms, we are dealing with the intersection of an osculating $5$-space to the $4$-th Veronese surface $S\subset\Ps^{14}$ and a secant $5$-space. But, again by a dimension count, one expects plenty of special choices of $l_1,\ldots ,l_6$ leading to \eqref{G}. This fits into well-known considerations about rank stratifications, in particular: the generic rank for ternary quartics is expected to be $5$, but it is actually $6$, because $(4,3)$ is an exceptional case of the Alexander-Hirschowitz theorem (see \cite[Theorem~3.2.2.4]{L}).

The expectation of a lot of solutions of \eqref{G} could be threatened by the existence of special quartics of higher rank than the generic. More precisely, if the conic represented by $q$ is nondegenerate and tangent to the line $x_2=0$ then the rank of $x_2^2q$ over the algebraic closure of $\K$ is strictly greater than $6$ (see \cite[Theorem 10.9.2.1]{L} or \cite[case (1) in the proof of Theorem~44, (p.~50)]{BGI}). If this was the case \emph{for all} quadratic forms $q\ne 0$, then \eqref{G} would admit no solutions, apart from those with $q=0$. Now, trivially, when $q=x_2^2$ the rank of $x_2^2q$ is one. Even if $q=x_2l$, with $l$ a linear form not proportional to $x_2$, then $qx_2^2$ is of rank $4$, at least over $\K=\mathbb{C}$ (see, e.g., \cite[Proposition~10.9.1.1]{L}). Incidentally, in such cases we do not immediately get a solution of \eqref{G}, because $l_1,\ldots ,l_6$, though generic, are in the form \eqref{l1-6}. The following example gives an explicit solution with $q$ of a more general type instead.

\begin{exm}
Set
\begin{multline*}
l_1=x_0\,,\;l_2=x_0+x_2\,,\;l_3=x_0-x_2\,,\\l_4=x_0+x_1\,,\;l_5=x_0+x_1+x_2\,,\;l_6=x_0+x_1-x_2\;.
\end{multline*}
We have
\begin{equation}
2l_1^4-l_2^4-l_3^4+2l_4^4-l_5^4-l_6^4=-4\left(6 x_0^2 + 6x_0x_1 + 3x_1^2 + x_2^2\right)x_2^2
\end{equation}
and $6 x_0^2 + 6x_0x_1 + 3x_1^2 + x_2^2=0$ is nondegenerate and not tangent to $x_2=0$.
\end{exm}

The above decomposition looks somewhat special, because the set of $h_i$'s reduces to only two values. That is why we find a bit surprising that this is a necessary condition, as we quickly explain now.
\begin{description}
\item[Claim]When the conic $q=0$ is nondegenerate and not tangent to the line $x_2=0$, for every decomposition \eqref{G} the $\alpha_i$'s are all nonzero and the $h_i$'s reduce to only two values.
\end{description}
To see this, let us split $l_i=L_i+k_ix_2$, with $L_i:=x_0+h_ix_1$, think of $L_1,\ldots,L_6$ as fixed and look for appropriate values of $\alpha_1,\ldots ,\alpha_6,k_1,\ldots, k_6$. We find ourselves dealing with a system of equations
\begin{equation}\label{S}
\sum_{i=1}^6\alpha_ih_i^d=0\,,\; 0\le d\le 4\;,\qquad\sum_{i=1}^6\alpha_ik_ih_i^d=0\,,\; 0\le d\le 3\;.
\end{equation}

Suppose that $h_1,\ldots, h_6$ are distinct. The left-hand equations in \eqref{S} involve some Vandermonde determinants, so that we have a solution $\alpha=\left(\alpha_1,\ldots, \alpha_6\right)$, with all nonzero entries $\alpha_i$, which is unique up to a scalar factor. The subsystem with $0\le d\le 3$ admits a two-dimensional space of solutions $\beta=\left(\beta_1,\ldots, \beta_6\right)$. Solutions $k=\left(k_1,\ldots, k_6\right)$ are given by $k_i=\beta_i/\alpha_i$, and therefore they form a two-dimensional space. But since
\begin{equation}\label{L}
\sum_{i=1}^6\alpha_iL_i\left(x_0,x_1\right)^4=0\;,
\end{equation}
for each choice of $k$ such that $l_i=L_i\left(x_0+\varrho_0x_2,x_1+\varrho_1 x_2\right)$ we have that \eqref{G} holds with $q=0$. Hence we also have a two-dimensional space of solutions $k$ for which $q=0$. This shows that no nonzero quartics of type $x_2^2q$ can arise when the $h_i$'s are distinct.

Suppose then that $h_1,\ldots, h_6$ are not distinct. Let us group together the $l_i$'s with the same $h_i$'s, and call $f_1,\ldots,f_s$ ($s\le 5$) the corresponding linear combinations of their fourth powers, with coefficients the $\alpha_i$'s.  Since less than six distinct $L_i^4$'s are linearly independent, in view of \eqref{L} we have that each $f_j$ contains $x_2$ as a factor, and it vanishes if it consists of only one term as a linear combination $\alpha_il_i^4$  (simply because $\alpha_i$ must vanish). Then, dividing by $x_2$ each nonzero $f_j$, we get $g_1,\ldots ,g_t$ such that
\begin{equation}\label{CG}
x_2q=g_1+\ldots+g_t
\end{equation}
and $1\le t\le 3$. Each $g_j$ must represent a cone (union of lines over the algebraic closure of $\K$) with vertex $V_j$ of the form $(h_i,-1,0)$ for some $i$ depending on $j$. Moreover, each $g_j$ either contains the line $x_2=0$, or intersects it into $3V_j$. Since the $V_j$'s are distinct, in view of \eqref{CG} we deduce that each $g_j$ contains $x_2$ as a factor. But when $f_j$ is a linear combination of only two of the $l_i$'s, then it can not contain $x_2^2$ as a factor (from an algebrogeometric viewpoint, that is an elementary fact about linear series on a line; it may also be deduced from \cite[Theorem~9.2.1.4]{L}). This proves that $t\le 2$. On the other hand, it can not be $t=1$ because $q$ does not represent a cone. Moreover, each one of the two nonzero $f_j$'s must be a linear combination of $3$ (and not less) of the $l_i^4$'s, and this immediately leads to the claim.

Following from the above line of thought, we turn our attention to a quartic $x_2^2q$, with $q$ representing a nondegenerate conic tangent to the line $x_2=0$. As mentioned before, it is known that such a quartic is of rank $7$, at least when $\K$ is algebraically closed. By some trial calculations, one can find $\alpha_1,\ldots ,\alpha_7$ and $l_1,\ldots, l_7$ such that $\alpha_1l_1^4+\ldots +\alpha_7l_7^4$ gives a quartic of that type, where $l_7=x_0+h_7x_1+k_7x_2$, similarly to the preceding $l_1,\ldots ,l_6$.

What is new in this case is that $h_1,\ldots ,h_7$ can be distinct. In our opinion, it would be reasonable to expect that a linear combination of $l_1^4,\ldots ,l_7^4$ that gives $x_2^2q$ should generically give a conic not tangent to the line $x_2=0$, and only in special cases a tangent one. The main result we are going to prove asserts that this expectation fails.

Before going into the proof, we want to briefly discuss some computational aspects. Put into elementary terms as before, our result reduces to the following assertion. If the $h_i$'s are distinct and $k_i$, $\alpha_i$ such that
\begin{equation}\label{FS}
\sum_{i=1}^7\alpha_ih_i^d=0\;,\; 0\le d\le 4\;,\qquad\sum_{i=1}^7\alpha_ik_ih_i^d=0\,,\; 0\le d\le 3\;,
\end{equation}
then we have
\[
\left(\sum_{i=1}^7\alpha_ik_i^2h_i\right)^2-\left(\sum_{i=1}^7\alpha_ik_i^2\right)\left(\sum_{i=1}^7\alpha_ik_i^2h_i^2\right)=0\;.
\]
In principle, this can be proved by a brute-force calculation. Indeed, let $f$ be the above polynomial in $h_i$, $k_i$, $\alpha_i$, let $g:=\prod_{j>i}\left(h_i-h_j\right)$
and $\mathfrak{a}$ the ideal generated by all polynomials in \eqref{FS}. Then it would suffice to check that $fg\in\sqrt\mathfrak{a}$. We tried to perform this checking with \cocoa\ (see \cite{C}) on a common computer. But, even with the simpler (sufficient) condition $fg\in\mathfrak{a}$, and even with some of the indeterminates specialized, the calculation was out of reach. Only some tests with many specializations ended up (with a positive answer).

\section{The main result}

The symmetric algebra $S^\bullet V$ of a $\K$-vector space $V$ will be denoted by $S_V$. The projective space $\Ps V$ will simply be the set of proportionality classes of nonzero vectors in $V$. An $f\in S_{V^\ast}$ will be interpreted, as usual, as a polynomial function on $V$. Dually, elements of $S_V$ are interpreted as polynomial functions on $V^\ast$, and we find it comfortable to denote the value of $s\in S_V$ on $x\in V^\ast$ by
\[
x(s)
\]
(for instance, with $v\in V$, we allow ourselves to say that $x\left(v^n\right)=x(v)^n=x^n(v)$).

To speed up calculations, we assume the following reasonable conventions. When dealing with $n$-tuples of polynomials, say $f=\left(f_1,\ldots ,f_n\right)$, $g=\left(g_1,\ldots ,g_n\right)$, we multiply them by the rule
\[
fg=\left(f_1g_1,\ldots,f_ng_n\right)
\]
(Hadamard product): it is nothing but the multiplication in the ordinary cartesian product ring $\left(S_{V^\ast}\right)^n$ (or $\left(S_V\right)^n$). We shall also make use of the standard bilinear form $\left(S_{V^\ast}\right)^n\times\left(S_{V^\ast}\right)^n\to S_{V^\ast}$:
\[
f\cdot g=f_1g_1+\cdots +f_ng_n\;,
\]
for which we shall keep the dot notation. Note that $fg\cdot h=f\cdot gh$. These operations can be performed, in particular, on elements of $\K^n\subseteq\left(S_{V^\ast}\right)^n$. Since we are considering $\K^n$ also as a ring, sometimes $1$ will stand for the identity element in this ring, i.e., $(1,\ldots,1)$. This way, a Waring decomposition of $f\in S^dV^\ast$ may be written
\[
f=1\cdot l^d\;,
\]
with $l\in\left(V^\ast\right)^n\subset\left(S_{V^\ast}\right)^n$ being an $n$-tuple of linear forms.

Let us also recall the notation $\overline{f}$ for the polarization of $f\in S^dV^\ast$, i.e., the symmetric $d$-multilinear form such that
\[
f(v)=\overline{f}(v,\ldots ,v)
\]
($d!\overline{f}$ may be regarded as the image of $f$ through a canonical map $S^dV^\ast\to\left(S^dV\right)^\ast$).
Moreover, \emph{partial polarizations of $f$},
\[
f_{\delta,d-\delta}:S^{\delta}V\to S^{d-\delta} V^\ast
\]
are also defined (see \cite[2.6.6]{L}; cf.~also \cite[Definition 17]{BGI}).

\begin{rem}\label{Pp}
Let $W$ be a $\K$-vector space and $\ell\in W^\ast$. The partial polarization \[\ell^{d+\delta}_{d,\delta}:S^dW\to S^\delta W^\ast\] of $\ell^{d+\delta}\in S^{d+\delta}W^\ast$ can be described by
\[
s\;\overset{\ell^{d+\delta}_{d,\delta}(t)}{\mapsto}\;\ell(ts)=\ell(t)\ell(s)\;,\qquad\forall t\in S^dW\;.
\]
Since partial polarization $f\mapsto f_{d,\delta}$ is a $\K$-linear procedure, if $L\in\left(W^\ast\right)^n$ is an $n$-tuple of linear forms, then for all $a\in\K^n$, the partial polarization $S^dW\to S^\delta W^\ast$ of $a\cdot L^{d+\delta}$ acts as follows:
\[
s\mapsto a\cdot L(t)L\left(s\right)=aL(t)\cdot L\left(s\right)\;,\qquad\forall t\in S^d W\;.
\]
\end{rem}

\begin{prop}\label{Mr}
Let $V$ be a $\K$-vector space, $\dim V=3$, and let $x\in V^\ast$, $q\in S^2V^\ast$, $x,q\ne 0$. If \[x^2q=l_1^4+\cdots +l_7^4\] with $l_1,\ldots,l_7\in V^\ast$ such that the lines $l_1=0$, $\ldots$, $l_7=0$ in $\Ps V$ intersect $x=0$ in seven distinct points, then the line $x=0$ is (possibly improperly) tangent to the conic $q=0$.
\end{prop}
\begin{proof}
Let $W:=x^\perp=\ker x\subseteq V$ and let us use capital letters for restrictions to~$W$:
\[
L_1:=l_1\restriction_W,\ldots ,\;L_7:=l_7\restriction_W\;.
\]
Set also $l:=\left(l_1,\ldots ,l_7\right)\in\left(V^\ast\right)^7\subset\left(S_{V^\ast}\right)^7$ and, similarly, $L:=\left(L_1,\ldots ,L_7\right)\in\left(S_{W^\ast}\right)^7$. By the hypothesis on the intersections with $\Ps W:x=0$ we have that $\left[L_1\right],\ldots ,\left[L_7\right]$ are distinct in $\Ps W^\ast$. This easily implies that $L_1^d,\ldots, L_7^d$ span $S^dW^\ast$ when $d\le 6$, so that in this case the linear map
\[
\begin{array}{cccc}
\varphi_d:&\K^7&\to&S^dW^\ast\\
&a&\mapsto& a\cdot L^d
\end{array}
\]
is surjective. Therefore we have
\[
\dim\ker\varphi_d=6-d\;,\qquad d=0,\ldots ,6\;.
\]
By \autoref{Pp}, and taking into account that $L\left(w^d\right)=L^d(w)$ for all $w\in W$, we have that if $a\in\ker\varphi_{d+\delta}$ and $t\in S^\delta W$ then $aL(t)\in\ker\varphi_d$. Thus, whenever $a\in\ker\varphi_{d+\delta}$, we can define a linear map
\[
\begin{array}{ccccl}
\psi_{a,\delta,d}:&S^\delta W&\to&\ker\varphi_d&\\
&t&\mapsto& aL(t)
\end{array}\;.
\]
Again because of the hypothesis on the intersections we have:
\begin{itemize}
\item a generator for $\ker\varphi_5$ is invertible in $\K^7$ (i.e., all of its components are nonzero);
\item with an invertible $a\in\K^7$ and $\delta\le 6$, the map $\psi_{a,\delta,d}$ is injective.
\end{itemize}
In particular, if $a$ is a generator for $\ker\varphi_5$ then the map $\psi_{a,1,4}$ is injective. It is henceforth surjective, because $\dim W=\dim\ker\varphi_4=2$. But since $1\cdot l^4=x^2q$, we have that $1\in\ker\varphi_4$. This way we end up with a $w\in W$ such that
\[
a=\frac1{L(w)}
\]
generates $\ker\varphi_5$.

Now let us pick $v\in V$ such that $x(v)=1$ and set $l_v:=l-l(v)x$, where
\[
l(v)x:=\left(l_1(v)x,\ldots ,l_7(v)x\right)\in\left(V^\ast\right)^7\subset\left(S_{V^\ast}\right)^7
\]
(in other terms, we are writing to the right the $S_{V^\ast}$--module multiplication in $\left(S_{V^\ast}\right)^7$, to avoid ambiguities due to the evaluation at $v$). Clearly $l_v(v)=0$ and the restriction of $l_v$ to $W$ is again $L$. From
\begin{equation}\label{Mf}
1\cdot\left(l_v+l(v)x\right)^4=1\cdot l^4=x^2q
\end{equation}
we deduce that
\[
1\cdot l_v^4+4\left(l(v)\cdot l_v^3\right)x\in x^2S^2V^\ast\subset S^4V^\ast\;.
\]
But from $1\cdot L^4=0$, $l_v(v)=0$ easily follows $1\cdot l_v^4=0$, hence $l(v)\cdot l_v^3$ is divisible by~$x$. Therefore
\[
l(v)\cdot L^3=0\;,
\]
that is, $l(v)\in\ker\varphi_3$. Since the injective linear map $\psi_{1/L(w),2,3}:S^2W\to\ker\varphi_3$ is also surjective by dimension reasons, we end up with a $b\in S^2W$ such that
\[
l(v)=\frac{L(b)}{L(w)}\;.
\]
With $Q:=q\restriction_W$ and $\overline{Q}$ its polarization, \eqref{Mf} implies
\[
Q=6l(v)^2\cdot L^2=6\frac{L(b^2)}{L(w)^2}\cdot L^2\;.
\]
Therefore, for all $u\in W$ we have
\[
\overline{Q}\left(w,u\right)=6\frac{L(b^2)}{L(w)}\cdot L(u)=0\;,
\]
since $L(b^2)/L(w)=\psi_{1/L(w),4,1}\left(b^2\right)\in\ker\varphi_1$. In other words, $w$ is in the kernel of the polarization $\overline{Q}$, which exactly means that $q=0$ and $x=0$ are tangent at~$[w]$.
\end{proof}

\begin{rem}
For fields of characteristic $\ge 7$ the proof basically works without changes (one has only to be careful with definitions about symmetric powers and polarizations). In characteristic $5$, $\ker\varphi_5$ fails to be one-dimensional, but we can still find an invertible element in it. In characteristic $2$ or $3$ the result is rather trivial, because $q$ becomes divisible by $x$.
\end{rem}

Though we have chosen quite a direct algebraic language, we prefer not to miss hinting at a more geometric interpretation, in view of possible generalizations. We refer readers not acquainted with algebraic geometry to \cite[Part~4]{L} for some good examples of similar techniques. In notation of the proof of \autoref{Mr}, let $Y:=\left\{\left[L_1\right],\ldots,\left[L_7\right]\right\}\subseteq\Ps W^\ast$. The maps $\varphi_d$ are dual to the restriction maps
\[
H^0\left(\mathcal{O}(d)\right)\to H^0\left(\mathcal{O}_Y(d)\right)\;,
\]
and hence their kernels are dual to $H^1\left(\mathcal{I}_Y(d)\right)$ (one might also note that they are naturally isomorphic to $\operatorname{Hom}(\mathcal{I}_Y(d),\omega)$, by Serre duality). The direct sum over $d$ of the $H^1\left(\mathcal{I}_Y(d)\right)$'s is known as the \emph{Hartshorne-Rao module of $Y$}, and the maps $\psi_{a,\delta,d}$ are byproducts of that module structure.

\bibliographystyle{plain}

\end{document}